\documentclass[a4paper, 11pt, reqno]{amsart}
\usepackage{amsmath,amsfonts,amssymb,amsthm,enumerate}
\usepackage{graphicx}
\usepackage{xy} \xyoption{all}

\newtheorem{thm}{Theorem}[section]
\newtheorem{cor}[thm]{Corollary}

\newtheorem{lem}[thm]{Lemma}

\theoremstyle{definition}
\newtheorem{defn}[thm]{Definition}
\newtheorem{exas}[thm]{Examples}
\newtheorem{rem}[thm]{Remark}

\let\phi\varphi

\begin{document}
\title{A criterion for Leavitt path algebras having invariant basis number}
\maketitle
\begin{center}
T.\,G.~Nam\footnote{Institute of Mathematics, VAST, 18 Hoang Quoc Viet, Cau Giay, Hanoi, Vietnam. E-mail address: \texttt{tgnam@math.ac.vn}} and N.\,T.~Phuc\footnote{Department of IT and Mathematics Teacher Training, Dong Thap University, Vietnam. E-mail address: \texttt{ntphuc@dthu.edu.vn}

{\bf Acknowledgements}: The authors take an opportunity to express their deep gratitude to Prof. Gene Abrams (Department of Mathematics, University of Colorado, Colorado Springs, Coloralo, USA) for his valuable suggestions in order to give the final shape of the paper.}
\end{center}

\begin{abstract} In this paper, we give a matrix-theoretic criterion for the Leavitt path algebra of
a finite graph has Invariant Basis Number. Consequently, we show that the Cohn path algebra of a finite graph
has Invariant Basis Number, as well as provide some certain classes of finite graphs for which
Leavitt path algebras having Invariant Basis Number.
\medskip

\textbf{Mathematics Subject Classifications}: 16S99, 18G05, 05C25

\textbf{Key words}: Cohn path algebra; Invariant Basis Number; Leavitt path algebra.
\end{abstract}

\section{Introduction}
The concept of projective modules over rings is a generalization of the idea of a
vector space; and their structure theory, in some sense, may
be considered as a generalization of the theorem asserting the existence and
uniqueness of cardinalities of bases for vector spaces. Projective modules play an important role in
different branches of mathematics, in particular, homological algebra and algebraic
K-theory. In general ring theory it is often convenient to impose certain conditions on
the projective modules, either to exclude pathological cases or to ensure better behaviour.
For  rings we have the following successively more restrictive conditions
on the projective (and in particular the free) modules:
\begin{enumerate}
\item[(1)] Invariant Basis Number (for short, IBN),
\item[(2)] Unbounded Generating Number,
\item[(3)] stably finite,
\item[(4)] the Hermite property (in P. M. Cohn's sense),
\item[(5)] cancellation of projectives.
\end{enumerate}
It is easily verified that each of these conditions is left-right symmetric and
entails the previous ones; moreover, in general, all these classes are distinct.

The conditions (1) - (3) occurred frequently among  the hypotheses in theorems
about rings, both in algebra and topology, many years ago. For examples of
the condition (3), see \cite{cs:othtoctc} and the references given there.
For basic properties of rings with these three conditions we may refer to
\cite{c:srotibp} and \cite{hv:ibnarpfr}. By finding conditions for an embedding of a (non-commutative)
ring in a skew field to be possible, P. M. Cohn has discovered the theory of free
ideal rings, and the conditions (1) - (5) above play an important role in this
theory (see, e.g., \cite{c:firaligr}). It is not at all easy to decide whether a given
ring has any one of these properties.

Given a (row-finite) directed graph $E$ and a field $K$, Abrams and Aranda Pino in
\cite{ap:tlpaoag05}, and independently Ara, Moreno, and Pardo in \cite{amp:nktfga},
introduced the \emph{Leavitt path algebra} $L_K(E)$. These Leavitt path algebras generalize
the Leavitt algebras $L_K(1, n)$ of \cite{leav:tmtoar}, and also contain many other interesting
classes of algebras. In addition, Leavitt path algebras are intimately related to graph
$C^*$-algebras (see \cite{r:ga}). In \cite{ag:lpaosg} Ara and Goodearl introduced and
investigated the \emph{Cohn path algebra} $C_K(E)$ of $E$ having coefficients in a field $K$.
Furthermore, in \cite[Theorem 1.5.17]{aam:lpa} Abrams, Ara and Siles Molina showed a perhaps-surprising
connection between Cohn and Leavitt path algebras in that every Cohn path algebra is, in fact, a Leavitt path algebra.

Recently, Kanuni and the first author \cite{ak:cpahibn} have showed that $C_K(E)$ has
IBN for every finite graph $E$. And the authors \cite{anp:lpahugn} have completely
classified those graphs $E$ for which $L_K(E)$ satisfies properties (2), (3), (4) and (5).
On the other hand, as of the writing of this article, it is an open question to give
graph-theoretic conditions on $E$ which describe precisely the Leavitt path algebras $L_K(E)$ having the
IBN property. The main goal of this note is to give a necessary and sufficient condition for
the Leavitt path algebra $L_K(E)$ having the IBN property.

The article is organized as follows. For the remainder of this introductory section we
recall the germane background information. In Section 2 we give a criterion for
the Leavitt path algebra of a finite graph having Invariant Basis Number (Theorem 2.5).
Applying the obtained result, we may cover Abrams and Kununi's result cited above (Corollary 2.7),
as well as show that the IBN property is not a Morita invariant within the class of algebras arising
as a Leavitt path algebra (Corollary 2.9). In Section 3, we establish some algebraic analogs of
Arklint and Ruiz's results, which are given in \cite[Section 3]{ar:cocka} (Lemmas 3.4 and 3.7, and
Corollary 3.9). Consequently, we may reduce the question to source-free graphs (Theorem 3.10),
and give some graphical sufficient conditions for Leavitt path algebras having Invariant Basis Number
(Corollaries 3.11 and 3.12).

Throughout this note, all rings are nonzero, associative with identity and all modules are
unitary. The set of nonnegative integers is denoted by $\mathbb{N}$, the integers by $\mathbb{Z}$.

A (directed) graph $E = (E^0, E^1, s, r)$ (or shortly $E = (E^0, E^1)$)
consists of two disjoint sets $E^0$ and $E^1$, called \emph{vertices} and \emph{edges}
respectively, together with two maps $s, r: E^1 \longrightarrow E^0$.  The
vertices $s(e)$ and $r(e)$ are referred to as the \emph{source} and the \emph{range}
of the edge~$e$, respectively. The graph is called \emph{row-finite} if
$|s^{-1}(v)|< \infty$ for all $v\in E^0$. All graphs in this paper will be assumed
to be row-finite. A graph $E$ is \emph{finite} if both sets $E^0$ and $E^1$ are finite
(or equivalently, when $E^0$ is finite, by the row-finite hypothesis).
A vertex~$v$ for which $s^{-1}(v)$ is empty is called a \emph{sink}; a vertex~$v$ for which
$r^{-1}(v)$ is empty is called a \emph{source}; a vertex~$v$ is calles an \emph{isolated vertex}
if it is both a source and a sink; and a vertex~$v$ is \emph{regular}
iff $0 < |s^{-1}(v)| < \infty$. A graph $E$ is said to be \emph{source-free} if it has no sources.

A \emph{path} $p = e_{1} \dots e_{n}$ in a graph $E$ is a sequence of
edges $e_{1}, \dots, e_{n}$ such that $r(e_{i}) = s(e_{i+1})$ for $i
= 1, \dots, n-1$.  In this case, we say that the path~$p$ starts at
the vertex $s(p) := s(e_{1})$ and ends at the vertex $r(p) :=
r(e_{n})$, and has \emph{length} $|p| := n$. We denote by $p^0$
the set of its vertices, that is, $p^0 = \{s(e_i), r(e_i)\ |\ i = 1,...,n\}$.
A path $p$ is called a \emph{cycle} if $s(p) = r(p)$, and $p$ does not revisit
any other vertex. A cycle $c$ is called a \emph{source cycle} if $|r^{-1}(v)| =1$
for all $v\in c^0$. A graph $E$ is \emph{acyclic} if it has no cycles.
An edge~$f$ is an \emph{exit} for a path $p = e_{1} \dots e_{n}$ if $s(f) = s(e_{i})$
but $f \ne e_{i}$ for some $1 \le i \le n$.

Let $E= (E^0, E^1)$ be a graph. For vertices $v, w\in E^0$, we write $v\geq w$ if there
exists a path in $E$ from $v$ to $w$, i.e., a path $p$ with $s(p) = v$ and $r(p) =w$.
Let $S$ be s subset of $E^0$. We write $v\geq S$ if there exists a $w\in S$ such that
$v\geq w$.

Let $H$ be a subset of $E^0$. The subset $H$ is called \emph{hereditary} if for all
$v\in H$ and $w\in E^0$, $v\geq w$ implies $w\in H$.

For any graph $E= (E^0, E^1)$, we denote by $A_E$ the \emph{incidence matrix} of
$E$. Formally, if $E^0 = \{v_1,..., v_n\}$, then $A_E = (a_{ij})$ the $n\times n$
matrix for which $a_{ij}$ is the number of edges having $s(e)=v_i$ and
$r(e)=v_j$. Specially, if $v_i\in E^0$ is a sink, then $a_{ij} = 0$ for all
$j=1,...,n$.

The notion of a Cohn path algebra has been defined and investigated by Ara and Goodearl \cite{ag:lpaosg}
(see, also, \cite{aam:lpa}). Specifically, for an arbitrary graph $E = (E^0,E^1,s,r)$
and an arbitrary field $K$, the \emph{Cohn path algebra} $C_{K}(E)$ of the graph~$E$
\emph{with coefficients in}~$K$ is the $K$-algebra generated
by the sets $E^0$ and $E^1$, together with a set of variables $\{e^{*}\ |\ e\in E^1\}$,
satisfying the following relations for all $v, w\in E^0$ and $e, f\in E^1$:
\begin{itemize}
\item[(1)] $v w = \delta_{v, w} w$;
\item[(2)] $s(e) e = e = e r(e)$ and
$r(e) e^* = e^* = e^*s(e)$;\item[(3)] $e^* f = \delta_{e, f} r(e)$.
\end{itemize}
Let $I$ be the ideal of $C_K(E)$ generated by all elements of the form
$v-\sum_{e\in s^{-1}(v)}ee^*$, where $v$ is a regular vertex.
Then the $K$-algebra $C_K(E)/I$ is called the \emph{Leavitt path algebra}
of $E$ with coefficients in $K$, denoted by $L_K(E)$.

Typically  the Leavitt path algebra $L_K(E)$ is defined without reference to Cohn
path algebras, rather, it is defined as the $K$-algebras are generated by
the set $\{v, e, e^*\ |\ v\in E^0, e\in E^1\}$,
which satisfy the  above conditions (1), (2), (3), and the additional condition:
\begin{itemize}
\item[(4)] $v= \sum_{e\in s^{-1}(v)}ee^*$ for any  regular vertex $v$.
\end{itemize}

If the graph $E$ is finite, both $C_K(E)$ and $L_K(E)$ are unital rings, each having
identity $1=\sum_{v\in E^0}v$ (see, e.g., \cite[Lemma 1.6]{ap:tlpaoag05}).

\section{A necessary and sufficient condition for Leavitt path algebras having invariant basis number}
In this section, we give a necessary and sufficient condition for
the Leavitt path algebra $L_K(E)$ of a finite graph $E$ with coefficients in
a field $K$ to have Invariant Basis Number. Consequently, we may get that Cohn path algebras
of finite graphs have Invariant Basis Number, which is established by Kanuni and the first
author \cite{ak:cpahibn}.

\begin{defn}
A ring $R$ is said to have \emph{Invariant Basis Number} (for short, \emph{IBN}) if,
for any pair of positive integers $m \text{ and } n$, $R^m\cong R^n$ (as right modules) implies that $m=n$.\hfill$\square$
\end{defn}

For any ring $R$ we denote by $\mathcal{V}(R)$ the set of isomorphism classes
(denoted by $[P]$) of finitely generated projective right $R$-modules, and we endow
$\mathcal{V}(R)$ with the structure of an abelian monoid by imposing the operation:
$$[P] + [Q] = [P\oplus Q]$$ for any isomorphism classes $[P]$ and $[Q]$. We note the following
easily verified equivalent characterizations of the IBN property.

\begin{rem}
The following conditions are equivalent for any ring $R$:

(1) $R$ has Invariant Basis Number;

(2) For any pair of positive integers $m \text{ and } n$, $m[R] = n[R]$ in $\mathcal{V}(R)$ implies that $m = n$;

(3) For any $A\in M_{m\times n}(R)$ and $B\in M_{n\times m}(R)$, if $AB= I_m$ and $BA=I_n$, then
$n= m$.\hfill$\square$
\end{rem}

One advantage of condition (3) in Remark 2.2 is that it involves neither left
nor right modules. In particular, the IBN property is indeed a left-right symmetric
condition in general.

The description of the monoid of isomorphism classes of finitely
generated projective modules of Leavitt path algebras which is due to Ara, Moreno and
Pardo \cite{amp:nktfga}. Namely, following \cite{amp:nktfga}, for any directed graph
$E=(E^0, E^1, s, r)$ we define the monoid $M_E$ as follows. We denote by $T$ the free abelian
monoid (written additively) with generators $E^0$. Define relations on $T$ by setting
\begin{center}
$v = \sum_{e\in s^{-1}(v)}r(e)\quad\quad\quad\quad\quad\quad\quad\quad\quad (M)$
\end{center}
for every regular vertex $v\in E^0$.
Let $\sim_{E}$ be the congruence relation on $T$ generated by these relations.
Then $M_E = T/_{\sim_E}$, and we also denote an element of $M_E$ by $[x]$,
where $x\in T$. In \cite[Theorem 3.5]{amp:nktfga} Ara, Moreno and Pardo
proved the following important result.

\begin{thm}[{\cite[Theorem 3.5]{amp:nktfga}}]
Let $E = (E^0, E^1)$ be a finite graph and $K$ an arbitrary field. Then the map
$[v]\longmapsto [vL_{K}(E)]$ yields an isomorphism of abelian monoids $M_E\cong \mathcal{V}(L_{K}(E))$.
In particular, under this isomorphism, we have $[\sum_{v\in E^0}v]\longmapsto [L_{K}(E)]$.
\end{thm}

Applying Theorem 2.3 and Remark 2.2(2), we immediately get the following corollary, which
provides us with a criterion to check the IBN property of $L_K(E)$ in terms of the monoid $M_E$.

\begin{cor}\label{cor4.2}
Let $E = (E^0, E^1)$ be a finite graph and $K$ any field.
Then the following conditions are equivalent:

(1) $L_K(E)$ has Invariant Basis Number;

(2) For any pair of positive integers $m$ and $n$,
\begin{center}
if $\ m[\sum_{v\in E^0}v] = n[\sum_{v\in E^0}v]$ in $M_E$, then $m = n$.
\end{center}
\end{cor}

We are now in position to give a necessary and sufficient condition for the Leavitt path algebra
of a finite graph to have Invariant Basis Number. To do so, we recall an important property of the
monoid $M_E$ as follows. Let $E$ be a finite graph having $|E^0| = h$, and regular (i.e., non-sink)
vertices $\{v_i\ |\ 1\leq i\leq z\}$. For $x = n_1v_1 + ... + n_hv_h\in T$ (the free abelian monoid
on generating set $E^0$), and $1\leq i\leq z$, let $M_i(x)$ denote the
element of $T$ which results by applying to $x$ the relation $(M)$ corresponding to vertex
$v_i$. For any sequence $\sigma$ taken from $\{1, 2,...,z\}$, and any $x\in T$, let $\Lambda_{\sigma}(x)\in T$
be the element which results by applying $M_i$ operations in the order specified by $\sigma$.\medskip

\noindent \textbf{The Confluence Lemma.} (\cite[Lemma 4.3]{amp:nktfga}) For each pair $x,\ y\in T$, $[x] = [y]$
in $M_E$ if and only if there are two sequences $\sigma$ and $\sigma'$ taken from $\{1, 2,...,z\}$
such that $\Lambda_{\sigma}(x)= \Lambda_{\sigma'}(y)$ in $T$.

\begin{thm}
Let $E$ be a finite graph having vertices $\{v_i\ |\ 1\leq i\leq h\}$ such that the regular
vertices appear as $v_1,..., v_z$. Let 
\[J_E = \left(
    \begin{array}{cc}
      I_z & 0 \\
      0 & 0 \\
    \end{array}
  \right)
\in M_{h}(\mathbb{N})\ \text{ and }\ b= [1 \ ...\ 1]^t
\in M_{h\times 1}(\mathbb{N}),\] and $[A^t_E - J_E \ \ b]$ the matrix gotten from
the matrix $A^t_E - J_E$ by adding the column $b$. Let $K$ be an arbitrary field.
Then $L_K(E)$ has Invariant Basis Number if and only if \[\mathrm{rank}(A^t_E - J_E)<
\mathrm{rank}([A^t_E - J_E \ \ b]).\]
\end{thm}
\begin{proof}
$(\Longleftarrow )$. Assume that $\mathrm{rank}(A^t_E - J_E)<\mathrm{rank}([A^t_E - J_E \ \ b])$;
we prove that $L_K(E)$ has Invariant Basis Number. We use Corollary 2.4 to do so. Namely, let $m$ and $n$
be positive integers such that
\begin{center}
$m[\sum_{i=1}^hv_i]=n[\sum_{i=1}^hv_i]$ in $M_E$.
\end{center}
We must show that $m = n$. By the Confluence Lemma and the hypothesis
$m[\sum_{i=1}^hv_i]=n[\sum_{i=1}^hv_i]$, there are two sequences
$\sigma$ and $\sigma'$ for which
$$\Lambda_{\sigma}(m\sum_{i=1}^hv_i)=\gamma=\Lambda_{\sigma'}(n\sum_{i=1}^hv_i)$$
for some $\gamma\in T$. But each time a substitution of the form
$M_j \ (1\leq j\leq z)$ is made to an element of $T$, the effect on
that element is to:
\begin{itemize}
\item[(i)] subtract $1$ from the coefficient on $v_j$;
\item[(ii)] add $a_{ji}$ to the coefficient on $v_i$ (for $1\leq i\leq h$).
\end{itemize}
For each $1\leq j\leq z$, denote the number of times that $M_j$ are invoked in
$\Lambda_{\sigma}$ and $\Lambda_{\sigma'}$ by $k_j$ and $k'_j$,
respectively. Recalling the previously observed effect of $M_j$ on
any element of $T$, we see that
\begin{equation*}
\begin{array}{rcl}
\gamma &=& \Lambda_{\sigma}(m\sum_{i=1}^hv_i)\\
&=& ((m-k_1)+ a_{11}k_1+a_{21}k_2+...+a_{z1}k_z)v_1\\
&&+ ((m-k_2)+ a_{12}k_1+a_{22}k_2+...+a_{z2}k_z)v_2+...\\
&&+ ((m-k_z)+ a_{1z}k_1+a_{2z}k_2+...+a_{zz}k_z)v_z\\
&&+ (m+  a_{1(z+1)}k_1+a_{2(z+1)}k_2+...+a_{z(z+1)}k_z)v_{z+1} +...\\
&&+ (m+a_{1h}k_1+a_{2h}k_2+...+a_{zh}k_z)v_h .
\end{array}
\end{equation*}
On the other hand, we have
\begin{equation*}
\begin{array}{rcl}
\gamma &=& \Lambda_{\sigma'}(n\sum_{i=1}^hv_i)\\
&=& ((n-k'_1)+ a_{11}k'_1+a_{21}k'_2+...+a_{z1}k'_t)v_1\\
&& + ((n-k'_2)+ a_{12}k'_1+a_{22}k'_2+...+a_{z2}k'_z)v_2+...\\
&& + ((n-k'_z)+ a_{1z}k'_1+a_{2z}k'_2+...+a_{zz}k'_z)v_z\\
&& + (n +  a_{1(z+1)}k'_1 + a_{2(z+1)}k'_2\ +... + a_{z(z+1)}k'_z)v_{z+1} + ... \\
&& + (n+a_{1h}k'_1+a_{2h}k'_2+...+a_{zh}k'_z)v_h .
\end{array}
\end{equation*}For each $1\leq j\leq z$, define $m_i = k'_i - k_i$. Then from the above observations,
equating coefficients on the free generators $T$, we get the following system of equations:
\begin{equation}\label{1}
\left\{
\begin{array}{rcl}
m-n & = &(a_{11}-1)m_1+a_{21}m_2+...+a_{z1}m_z \\
m-n & = & a_{12}m_1+(a_{22}-1)m_2+...+a_{z2}m_z \\
& ... & \\
m-n & = & a_{1z}m_1+a_{2z}m_2+...+ (a_{zz}-1)m_z\\
m-n & = & a_{1(z+1)}m_1+a_{2(z+1)}m_2+...+a_{z(z+1)}m_z \\
 & ... & \\
m-n & = & a_{1h}m_1+a_{2h}m_2+...+a_{zh}m_z \\
\end{array}
\right.
\end{equation}
In other words, the $h$-tuple $(m_1,..., m_z, 0,..., 0) \in \mathbb{Z}^h$ is a solution of
the following linear system: \[(A^t_E - J_E)\mathrm{x} = (m-n)b,\] where $\mathrm{x}
= [x_1\ ...\ x_h]^t$ is the unknown vector. This implies that
\[\mathrm{rank}(A^t_E - J_E)= \mathrm{rank}([A^t_E - J_E \ \ (m-n)b]).\]

If $m-n\neq 0$, then we have obviously that \[\mathrm{rank}([A^t_E - J_E \ \ (m-n)b]) =
\mathrm{rank}([A^t_E - J_E \ \ b]),\] so $\mathrm{rank}(A^t_E - J_E)= \mathrm{rank}([A^t_E - J_E \ \ b]),$
contradicting the above hypothesis; which gives $m = n$. Therefore, $L_K(E)$ has Invariant
Basis Number.

$(\Longrightarrow)$. Assume conversely that $\mathrm{rank}(A^t_E - J_E)= \mathrm{rank}([A^t_E - J_E \ \ b])$.
We will prove that $L_K(E)$ does not have Invariant Basis Number, that means, we have to find
a pair of distinct positive integers $m$ and $n$ such that

$$m[\sum_{i=1}^hv_i]=n[\sum_{i=1}^hv_i]$$ in $M_E$. Equivalently, arguing as in the previous half
of the proof, we show that we can find a pair of distinct positive integers $m \text{ and } n$, and
nonnegative integers $k_j,\ k'_j\ (j= 1,..., z)$ such that
\begin{equation}\label{2}
\left\{
\begin{array}{rcl}
m-n & = &(a_{11}-1)m_1+a_{21}m_2+...+a_{z1}m_z \\
m-n & = & a_{12}m_1+(a_{22}-1)m_2+...+a_{z2}m_z \\
& ... & \\
m-n & = & a_{1z}m_1+a_{2z}m_2+...+ (a_{zz}-1)m_z\\
m-n & = & a_{1(z+1)}m_1+a_{2(z+1)}m_2+...+a_{z(z+1)}m_z \\
 & ... & \\
m-n & = & a_{1h}m_1+a_{2h}m_2+...+a_{zh}m_z \\
\end{array}
\right.
\end{equation}
where $m_j :=k'_j-k_j$ for all $j = 1, ..., z$.
It is customary to identify this system of linear equations with the
matrix-vector equation \[(A^t_E - J_E)\mathrm{x} = (m-n)b,\] where
$\mathrm{x} = [m_1,..., m_z, 0,..., 0]\in \mathbb{Z}^h$.

We now choose the above integers as follows: By $\mathrm{rank}(A^t_E - J_E)=
\mathrm{rank}([A^t_E - J_E \ \ b])=:r\leq z$, after finite numbers of
elementary row transformations, $[A^t_E - J_E \ \ b]$ can be brought to the form:

$$\left(\begin{tabular}{cccccccccccccc}
0&...&0&$a_{1j_1}$&...&$a_{1(j_2-1)}$&0          &$a_{1(j_2+1)}$&...&$a_{1(j_r-1)}$&0          &...&$b_1$\\
0&...&0&0                      &...&0            &$a_{2 j_2}$&$a_{2(j_2+1)}$&...&$a_{2(j_r-1)}$&0          &...&$b_2$\\
.&.  &.&.                     &.  &.            &.          &.            &.  &.            &.          &.  &.  \\
0&...&0&0                     &...&0            &0          &0            &...&0            &$a_{rj_r}$&...&$b_r$\\
0&...&0&0                      &...&0            &0          &0            &...&0            &0          &...&0         \\
.&.  &.&.                     &.  &.            &.          &.            &.  &.            &.          &.  &.         \\
0&...&0&0                      &...&0            &0          &0            &...&0            &0          &...&0         \\
\end{tabular}\right)$$
where the entries are integers, $j_1<j_2<...<j_r$, $a_{1j_1}a_{2j_2}...a_{rj_r}\neq 0$ and
$\sum^r_{i = 1}b^2_i \neq 0$ (i.e., $b_i$'s are not equal to zero, simultaneously). We then choose the
integers $m_j$, $n$ and $m$ as follows:

\begin{equation*}
m_j :=  \left\{
\begin{array}{lcl}
\frac{b_j|a_{1j_1}a_{2j_2}...a_{rj_r}|}{a_{ij_i}}&  & \text{if } j = j_i\ (1\leq i\leq r) , \\
&  &  \\
0&  & \text{otherwise \ \ }%
\end{array}%
\right.
\end{equation*}%

$$n :=\max\{|m_j|\mid j = 1,...,h\} \text{ and } m:= |a_{1j_1}a_{2j_2}...a_{rj_r}| + n.$$

Finally, positive integers $k_j$ and $k'_j$ $(j = 1, ..., t)$ are
chosen by the rule:

\begin{equation*}
(k'_j, k_j) :=  \left\{
\begin{array}{lcl}
(0, 0)&  & \text{if } m_j = 0 , \\
&  &  \\
(m_j, 0)& & \text{if } m_j > 0 ,\\
& & \\
(0, - m_j)&  & \text{if } m_j < 0.%
\end{array}%
\right.
\end{equation*}%
Then, a tedious but straightforward computation yields that this
choice of integers indeed satisfies the system of equations (2) above.
Also, notice that we always have  $k'_j, k_j\leq n< m$ (for all $j = 1,..., h$),
so  the sequences $\sigma$ and $\sigma'$  which are produced this way can not lead
to a situation where we have some elements in the monoid $T$ where the coefficient
on a vertex is negative, thus completing the proof of the theorem.
\end{proof}

\begin{exas}
We present a specific example of the construction presented in the proof of Theorem 2.5
which shows that graphs satisfying $\mathrm{rank}(A^t_E - J_E)=\mathrm{rank}([A^t_E - J_E \ \ b])$
do not have IBN. Let $K$ be a field and let $E$ be the graph

$$
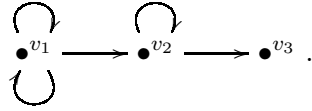
\xymatrix{\bullet^{v_1}\ar@(ul,ur) \ar@(dr,dl) \ar[r]& \bullet^{v_2} \ar@(ul,ur) \ar[r]& \bullet^{v_3}}.$$
\\ We then have

$$A^t_E=\left(\begin{tabular}{ccc}
2&0&0\\
1&1&0\\
0&1&0\\
\end{tabular}\right) \text{ and } J_E=\left(\begin{tabular}{ccc}
1&0&0\\
0&1&0\\
0&0&0\\
\end{tabular}\right),$$ so
$$[A^t_E - J_E \ \ b] = \left(
    \begin{array}{cccc}
      1 & 0 & 0 & 1 \\
      1 & 0 & 0 & 1 \\
      0 & 1 & 0 & 1 \\
    \end{array}
  \right).
$$
\end{exas}

Clearly, $\mathrm{rank}(A^t_E - J_E)=2=\mathrm{rank}([A^t_E - J_E \ \ b])$, and
$[A^t_E - J_E \ \ b]$ can be brought to the form:
\[\left(
    \begin{array}{cccc}
      1 & 0 & 0 & 1 \\
      0 & 1 & 0 & 1 \\
      0 & 0 & 0 & 0 \\
    \end{array}
  \right),\] so $j_1 = 1, j_2 = 2$, $a_{11} =1 = a_{22}$ and $b_1 =1=b_2$.

As in the proof of Theorem 2.5, we define $m_j\ (j = 1, 2)$, $n$ and $m$ as follows:
\[m_1 = \frac{b_1|a_{11}a_{22}|}{a_{11}} = 1,\ \ m_2 = \frac{b_2|a_{11}a_{22}|}{a_{22}} = 1\]
\[n= \max\{m_1, m_2\} = 1 \text{ and } m = |a_{11}a_{22}| + n = 2.\]
Subsequently, we define \[k'_1 = m_1 = 1 \text{ and } k_1 = 0\]
\[k'_2 = m_2 = 1 \text{ and } k_2 = 0.\] Then the construction described in the proof of
Theorem 2.5 yields \[2[\sum^3_{i=1}v_i] = [\sum^3_{i=1}v_i]\] in $M_E$. Equivalently,
we achieve this by verifying the equivalent version \[[2v_1 + 2v_2 + 2v_3] = [v_1 + v_2 + v_3]\]
in $M_E$. In $M_E$ we have:
\[(\mathrm{i}) \ \ [v_1] = [2v_1 + v_2],\ \  \text{ and } \ \  (\mathrm{ii}) \ \ [v_2] = [v_2 + v_3].\]
The right side can be transformed as follows:

\begin{equation*}
\begin{array}{rcl}
[v_1 + v_2 + v_3] &=& [2v_1 + v_2 + v_2 + v_3]\  \ \ \ \ \ \ \ \ \ \ \ \text{by} \ (i)\\
&=& [2v_1 + 2v_2 + v_3]\\
&=&[2v_1 + v_2 + v_2 +v_3+ v_3] \ \ \ \ \ \text{by}\  (ii)\\
&=&[2v_1 + 2v_2 + 2v_3].
\end{array}
\end{equation*}
This completes the verification that the two quantities are indeed equal in $M_E$, which shows
that $L_K(E)$ does not have Invariant Basis Number.\hfill$\square$\medskip

Now we provide a few remarks about Cohn path algebras.
We represent a specific case of a more general result described in \cite[Section 1.5]{aam:lpa}. Namely,
let $E = (E^0, E^1, s, r)$ be an arbitrary graph and $Y$ the set of regular vertices of $E$.
Let $Y'=\{v'\ |\ v\in Y\}$ be a disjoint copy of $Y$. For $v\in Y$ and for each edge $e$ in $E^1$ such
that $r_E(e)=v$, we consider a new symbol $e'$. We define the graph $F(E)$ as follows:
$$F(E)^0 := E^0\sqcup Y' \text{ and } F(E)^1 :=E^1\sqcup \{e'\ |\ r_E(e)\in
Y\},$$ and for each $e\in E^1$, $s_{F(E)}(e)=s_E(e),\ s_{F(E)}(e')=s_E(e),\
r_{F(E)}(e)=r_E(e)$, and $r_{F(E)}(e')=r_E(e)'$. For instance, if

$$E = \xymatrix{\bullet^{v}\ar@(ul,ur)^e \ar@(dr,dl)^f},\ \ \text{ then } \ \
F(E) = \xymatrix{\bullet^{v}\ar@(ul,ur)^e \ar@(dr,dl)^f \ar@/_.5pc/[r]_{f'} \ar@/^.5pc/[r]^{e'}& \bullet^{v'}}.$$
In \cite[Theorem 1.5.17]{aam:lpa}, Ara, Siles Molina and the first author showed that for any field $K$
and any graph $E$, there is an isomorphism of $K$-algebras $C_K(E)\cong L_K(F(E))$.

In \cite[Theorem 9]{ak:cpahibn} Kanuni and the first author showed that
the Cohn path algebra of a finite graph has Invariant Basis Number. Now we will provide
another proof for this interesting result in terms of Theorem 2.5.

\begin{cor}[{\textit{cf.}~\cite[Theorem 9]{ak:cpahibn}}]
Let $E$ be a finite graph and $K$ a field. Then $C_K(E)$ has Invariant Basis Number.
\end{cor}
\begin{proof}
We first have that $C_K(E)\cong L_K(F(E))$, where $F(E)$ is the graph described above.
We next prove that the graph $F(E)$ satisfies the condition that
\[\mathrm{rank}(A^t_{F(E)} - J_{F(E)})<\mathrm{rank}([A^t_{F(E)} - J_{F(E)} \ \ b]).\]
Indeed, we denote $E^0$ by $\{v_i\ |\ 1\leq i\leq h\}$, in such a way that the regular
vertices appear as $v_1,..., v_z$. We then have that $F(E)^0 = \{v_1,..., v_h, v'_1,..., v'_z\}$,
and the only regular vertices of $F(E)$ are $\{v_1, v_2,..., v_z\}$. Notice that $|F(E)^0| = h +z$.


Let $A_E = (a_{ij})_{h\times h}$ be the incidence matrix of $E$.
Then the incidence matrix of $F(E)$ is the $(h+z)\times (h+z)$-matrix:
$$A_{F(E)}=\left(\begin{tabular}{ccccccc}
$a_{11}$&$a_{12}$&...&$a_{1h}$&$a_{11}$&...&$a_{1z}$\\
$a_{21}$&$a_{22}$&...&$a_{2h}$&$a_{21}$&...&$a_{2z}$\\
.&.&.&.&.&.&.\\
$a_{z1}$&$a_{z2}$&...&$a_{zh}$&$a_{z1}$&...&$a_{zz}$\\
0&0&...&0&0&...&0\\
.&.&.&.&.&.&.\\
0&0&...&0&0&...&0\\
\end{tabular}\right),$$ and hence
$$[A^t_{F(E)} - J_{F(E)} \ \ b]=\left(\begin{tabular}{cccccccc}
$a_{11}-1$&$a_{21}$&...&$a_{z1}$&0&...&0&1\\
$a_{12}$&$a_{22}-1$&...&$a_{z2}$&0&...&0&1\\
.&.&.&.&.&.&.\\
$a_{1z}$&$a_{2z}$&...&$a_{zz} -1$&0&...&0&1\\
$a_{1(z+1)}$&$a_{2(z+1)}$&...&$a_{z(z+1)}$&0&...&0&1\\
.&.&.&.&.&.&.\\
$a_{1h}$&$a_{2h}$&...&$a_{zh}$&0&...&0&1\\
$a_{11}$&$a_{21}$&...&$a_{z1}$&0&...&0&1\\
$a_{12}$&$a_{22}$&...&$a_{z2}$&0&...&0&1\\
.&.&.&.&.&.&.\\
$a_{1z}$&$a_{2z}$&...&$a_{zz}$&0&...&0&1\\
\end{tabular}\right).$$
For each $1\leq i\leq z$, we subtract row $i$ from row $h+ i$ in the matrix
$[A^t_{F(E)} - J_{F(E)} \ \ b]$, which yields the equivalent matrix $B$:

$$B=\left(\begin{tabular}{cccccccc}
$a_{11}-1$&$a_{21}$&...&$a_{z1}$&0&...&0&1\\
$a_{12}$&$a_{22}-1$&...&$a_{z2}$&0&...&0&1\\
.&.&.&.&.&.&.\\
$a_{1z}$&$a_{2z}$&...&$a_{zz} -1$&0&...&0&1\\
$a_{1(z+1)}$&$a_{2(z+1)}$&...&$a_{z(z+1)}$&0&...&0&1\\
.&.&.&.&.&.&.\\
$a_{1h}$&$a_{2h}$&...&$a_{zh}$&0&...&0&1\\
1&0&...&0&0&...&0&0\\
0&1&...&0&0&...&0&0\\
.&.&.&.&.&.&.\\
0&0&...&1&0&...&0&0\\
\end{tabular}\right).$$
Next, we write $B$ in the form $B = (b_{ij})_{(h+z)\times (h+z +1)}$. If $b_{ij} \neq 0$
$(1\leq i\leq h, 1\leq j\leq z)$, then we subtract $b_{ij}$ times row $h+j$ from row $i$
in the matrix $B$, which yields the equivalent matrix $C$:

$$C=\left(\begin{tabular}{cccccccc}
$0$&$0$&...&$0$&0&...&0&1\\
$0$&$0$&...&$0$&0&...&0&1\\
.&.&.&.&.&.&.\\
$0$&$0$&...&$0$&0&...&0&1\\
$0$&$0$&...&$0$&0&...&0&1\\
.&.&.&.&.&.&.\\
$0$&$0$&...&$0$&0&...&0&1\\
1&0&...&0&0&...&0&0\\
0&1&...&0&0&...&0&0\\
.&.&.&.&.&.&.\\
0&0&...&1&0&...&0&0\\
\end{tabular}\right).$$
Then we immediately get that \[\mathrm{rank}(A^t_{F(E)} - J_{F(E)}) = z < z+1 =
\mathrm{rank}([A^t_{F(E)} - J_{F(E)} \ \ b]),\] which gives that $C_K(E)$ has
Invariant Basis Number, by Theorem 2.5.
\end{proof}

It is known that the IBN property is not a Morita equivalent property for rings (see, e.g.,
\cite[Exercise 11, page 502]{l:lomar}). As another application of Theorem 2.5, we may
construct such counterexamples where both of the rings are Leavitt path algebras. Before
doing this, we recall the following notion:

\begin{defn}[{\cite[Definition 1.2]{alps:fiitcofpa} and \cite[Notation 2.4]{ar:fpsmolpa}}]
Let $E = (E^0, E^1, r, s)$ be a graph, and let $v\in E^0$ be a source. We form the \emph{source
elimination} graph $E_{\setminus v}$ of $E$ as follows:
$(E_{\setminus v})^0 = E^0\setminus \{v\}$, $(E_{\setminus v})^1 =
E^1\setminus s^{-1}(v)$, $s_{E_{\setminus v}} = s|_{(E_{\setminus
v})^1}$ and $r_{E_{\setminus v}} = r|_{(E_{\setminus v})^1}$. In
other words, $E_{\setminus v}$ denotes the graph gotten from $E$ by
deleting $v$ and all of edges in $E$ emitting from $v$.\hfill$\square$
\end{defn}
Ara and Rangaswamy \cite[Lemma 4.3]{ar:fpsmolpa} have proved that if
$E$ is a finite graph, $v$ is a source which is not a sink, and $K$ is a field, then
$L_R(E)$ is Morita equivalent to $L_K(E_{\setminus v})$. Using this key note and
Theorem 2.5, we have the following:

\begin{cor}
The Invariant Basis Number property is not Morita invariant within the class of algebras arising
as a Leavitt path algebra.
\end{cor}
\begin{proof}
Let $K$ be a field, and let $E$ and $F$ be the graphs, respectively:

$$E=
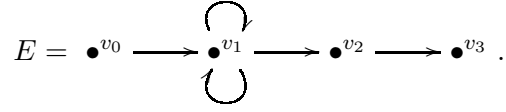
\xymatrix{\bullet^{v_0}\ar[r]&\bullet^{v_1}\ar@(ul,ur) \ar@(dr,dl) \ar[r]& \bullet^{v_2}\ar[r]& \bullet^{v_3}}.$$
and \bigskip
\medskip
$$F=
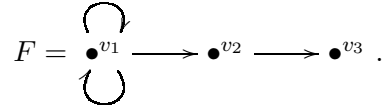
\xymatrix{\bullet^{v_1}\ar@(ul,ur) \ar@(dr,dl) \ar[r]& \bullet^{v_2}\ar[r]& \bullet^{v_3}}.$$
\\ We then clearly get that $F$ is the graph gotten from $E$ by the process of source elimination $v_0$,
and hence, $L_K(E)$ is Morita equivalent to $L_K(F)$,
by Ara and Rangaswamy's result \cite[Lemma 4.3]{ar:fpsmolpa}.
Also, we have that

$$[A^t_E -J_E\ \ b] =\left(
                       \begin{array}{ccccc}
                         -1 & 0 & 0 & 0 & 1 \\
                          1 & 1 & 0 & 0 & 1 \\
                         0 & 1 & -1 & 0 & 1 \\
                         0 & 0 & 1 & 0 & 1 \\
                       \end{array}
                     \right)
 \text{ and } [A^t_F-J_F \ \ b] =\left(
                               \begin{array}{cccc}
                                 1 & 0 & 0 & 1 \\
                                 1 & -1 & 0 & 1 \\
                                 0 & 1 & 0 & 1 \\
                               \end{array}
                             \right).$$
This implies
\[\mathrm{ramk}(A^t_E -J_E) = 3 = \mathrm{rank}([A^t_E -J_E\ \ b])\]
and \[\mathrm{ramk}(A^t_F -J_F) = 2< 3 = \mathrm{rank}([A^t_F -J_F\ \ b]).\] Therefore,
$L_K(F)$ has Invariant Basis Number, but $L_K(E)$ does not have Invariant Basis Number,
by Theorem 2.5.
\end{proof}

\section{Graphical sufficient conditions for Leavitt path algebras having Invariant Basis Number}
In this section, we show some certain classes of finite graphs for which Leavitt path algebras
having Invariant Basis Number by using Theorem 2.5. Before doing this, we establish some
algebraic analogs of Arklint and Ruiz's results, which are given in \cite[Section 3]{ar:cocka}.

\begin{defn}[{\cite[Definition 3.2]{ar:cocka}}]
Let $E$ be a graph, let $v_0\in E^0$ be a vertex, and let $n$ be a
positive integer. Define a graph $E(v_0, n)$ as follows:
\[E(v_0, n)^0 = E^0 \cup \{v_1, v_2, ..., v_n\}\]
\[E(v_0, n)^1 = E^1 \cup \{e_1, e_2, ..., e_n\}\]
where $r_{E(v_0, n)}$ and $s_{E(v_0, n)}$ extends $r_E$ and $s_E$ respectively and
$r_{E(v_0, n)}(e_i) = v_{i-1}$ and $s_{E(v_0, n)}(e_i) = v_{i}$.
\end{defn}

\begin{defn}[{\cite[Definition 3.3]{ar:cocka}}]
Let $E$ be a graph, let $e_0\in E^1$ be an edge, and let $n$ be a
positive integer. Define a graph $E(e_0, n)$ as follows:
\begin{equation*}
\begin{array}{l}
E(e_0, n)^0=E^0 \cup \{v_1, v_2, ..., v_n\}\\
E(e_0, n)^1=E^1\setminus\{e_0\} \cup \{e_1, e_2, ...,
e_{n+1}\}\\
\end{array}
\end{equation*}
where $r_{E(e_0, n)}$ and $s_{E(e_0, n)}$ extends $r_E$ and $s_E$ respectively and
$r_{E(e_0, n)}(e_i) = v_{i-1}$ for $i=2, ..., n+1$ and $s_{E(e_0, n)}(e_i) = v_{i}$ for
$i = 1, ..., n$, and $r_{E(e_0, n)}(e_1) = r_E(e_0)$ and $s_{E(e_0, n)}(e_{n+1}) = s_{E}(e_0)$.
\end{defn}

\begin{exas}
Let $E$ be the graph
$$
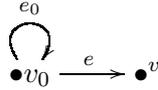
\xymatrix{\bullet{v_0} \ar@(ul,ur)^{e_0} \ar[r]^e& \bullet^v}$$ Then
$E(v_0,2)$ is the graph
$$
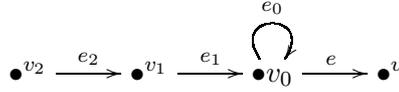
\xymatrix{\bullet^{v_2} \ar[r]^{e_2}& \bullet^{v_1} \ar[r]^{e_1} &
\bullet{v_0} \ar@(ul,ur)^{e_0} \ar[r]^e& \bullet^v}$$ and $E(e_0,2)$
is the graph
$$\xymatrix{&\bullet^{v_2} \ar[d]^{e_2}&&\\
& \bullet^{v_1} \ar[r]^{e_1} & \bullet{v_0} \ar[ul]_{e_3} \ar[r]^e&
\bullet^v}$$
\end{exas}

\begin{lem}[{cf. \cite[Proposition 3.5]{ar:cocka}}]
Let $K$ be a field and $E$ a graph, let $e_0\in E^1$ be an edge, and let $n$ be a positive integer.
Define $v_0 = r_{E}(e_0)$. Then $$L_K(E(v_0, n))\cong L_K(E(e_0, n)).$$
\end{lem}
\begin{proof} Let us consider an $K$-algebra homomorphism $$\varphi: L_K(E(v_0, n))\longrightarrow
L_K(E(e_0, n))$$ given on the generators of the free $K$-algebra $K\langle v, e, e^* \mid v\in E(e_0, n)^0,
e\in E(e_0, n)^1\rangle$ as follows: $\varphi(v) = v$

\begin{equation*}
\varphi(e)=  \left\{
\begin{array}{lcl}
e&  & \text{if } e\neq e_0 , \\
e_{n+1}e_n...e_1&  & \text{otherwise \ \ }%
\end{array}%
\right.
\end{equation*}%
and \begin{equation*}
\varphi(e^*)=  \left\{
\begin{array}{lcl}
e^*&  & \text{if } e\neq e_0 , \\
e^*_1...e^*_n e^*_{n+1}&  & \text{otherwise. \ \ }%
\end{array}%
\right.
\end{equation*}%
To be sure that in a such manner defined map $\varphi: L_K(E(v_0, n))\longrightarrow
L_K(E(e_0, n))$, indeed, provides us with the desired ring homomorphism, we only need to verify that all
following elements:

$vw - \delta_{v, w}v$ for all $v, w\in E(v_0,n)^0$,

$s_{E(v_0, n)}(e)e - e$ and $e-er_{E(v_0, n)}(e)$ for all $e\in E(v_0,n)^1$,

$r_{E(v_0, n)}(e)e^*-e^*$ and $e^*-e^*s_{E(v_0, n)}(e)$ for all $e\in E(v_0,n)^1$,

$e^*f - \delta_{e, f}r_{E(v_0, n)}(e)$ for all $e, f\in E(v_0,n)^1$,

$v - \sum_{e\in (s_{E(v_0, n)})^{-1}(v)}ee^*$ for a regular vertex $v\in E(v_0,n)^0$\\
are in the kernel of $\varphi$. But the latter can be established right away by repeating verbatim
the corresponding obvious arguments in the proof of \cite[Proposition 3.5]{ar:cocka}. Note that
the only generator of $L_K(E(e_0, n))$ that is not included in the generators of $L_K(E(v_0, n))$
is $e_{n+1}$. In this case, we note that we always have $v_i = e_ie^*_i$ for all $i = 1,..., n$, and
hence, \[\varphi(e_0e^*_1...e^*_n) = e_{n+1}e_n...e_1e^*_1...e^*_n = e_{n+1}.\] Therefore,
$e_{n+1} \in \varphi(L_K(E(v_0, n)))$, which implies that $\varphi$ is surjective.

We next prove that $\varphi$ is injective: Indeed, suppose $\varphi$ is not injective, that means,
we then have that $\ker(\varphi) \neq 0$. By \cite[Theorem 6]{col:tsiilpa}, $\ker(\varphi)$ contains
a nonzero element $\alpha$ of the form: \[\alpha = v + \sum_{i=1}^nk_ic^i,\] where $v\in E(v_0, n)^0$,
$c$ is a cycle in $E(v_0,n)$ based at $v$, and $k_i\in K$ for $1\leq i\leq n$. We consider the following two cases:

\emph{Case} 1. The cycle $c$ has an exit $f\in E(v_0, n)^1$, say $c:= f_1 ... f_m$. Then, there exists
$1\leq j\leq m$ such that $f\neq f_j$ and $s(f) = s(f_j)$, and hence,
$$z^*\alpha z = z^*vz + \sum_{i=1}^nk_iz^*c^iz = r(z)\in \ker(\varphi)$$ for $z:= f_1 ... f_{j-1}f$.
This implies that $r(z) = \varphi(r(z)) = 0$, a contradiction.

\emph{Case} 2. The cycle $c$ has no an exit. Note first that the cycle structure of $E(v_0,n)$ is determined by
the cycle structure of $E$ and vice versa. Moreover, the cycles of $E(v_0,n)$ without exits are
in one-to-one correspondence to the cycles of $E(e_0, n)$ without exits. By this note, we have that $p:= \varphi(c)$
is a cycle in $E(e_0, n)$ without exits based at $v$, and hence, $vL_K(E(e_0,n))v = K[p, p^*]$ is isomorphic
to the Laurent polynomial ring $K[x, x^{-1}]$, via an isomorphism that sends $v$ to $1$, $p$ to $x$
and $p^*$ to $x^{-1}$, by \cite[Lemma 2.2.1]{aam:lpa}. This implies that
$$0= \varphi(\alpha) = v + \sum_{i=1}^nk_ip^i \neq 0,$$ a contradiction.

From the two paragraphs above, we get immediately that $\varphi$ is injective. Therefore, $\varphi$ is
an isomorphism, finishing the proof.
\end{proof}

\begin{defn}[{\cite[Definition 3.6]{ar:cocka}}]
Let $E$ be a graph and let $H$ be a hereditary subset of $E^0$. Consider the set
\[F(H) =\{\alpha \mid \alpha = e_1e_2...e_n, s_E(e_n)\notin H, r_E(e_n)\in H\}.\] Let
$\overline{F}(H)$ be another copy of $F(H)$ and we write $\overline{\alpha}$ for the
copy of $\alpha$ in $\overline{F}(H)$. Define a graph $E(H)$ as follows:
\begin{equation*}
\begin{array}{l}
E(H)^0 = H\cup F(H)\\
E(H)^1 = s^{-1}_E(H)\cup \overline{F}(H)\\
\end{array}
\end{equation*} and extend
$s_E$ and $r_E$ to $E(H)$ by defining $s_{E(H)}(\overline{\alpha}) = \alpha$ and
$r_{E(H)}(\overline{\alpha}) = r(\alpha)$.
\end{defn}

Notice that $E(H)$ is just the graph $(H, s^{-1}_E(H), s_E, r_E)$ together with a
source for each $\alpha\in F(H)$ with exactly one edge from $\alpha$ to $r_E(\alpha)$.

\begin{exas}
Let $E$ be the graph
$$
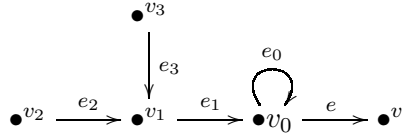
\xymatrix{&\bullet^{v_3} \ar[d]^{e_3}&&\\
\bullet^{v_2} \ar[r]^{e_2}& \bullet^{v_1} \ar[r]^{e_1} &
\bullet{v_0} \ar@(ul,ur)^{e_0} \ar[r]^e& \bullet^v}$$ and
$H=\{v_0, v\}$. Then $F(H)=\{e_1,e_2e_1,e_3e_1\}$. Therefore, the graph
$$\xymatrix{\bullet^{e_3e_1} \ar[dr]^{\overline{e_3e_1}}&& \\
\bullet^{e_1} \ar[r]^{\overline{e_1}}& \bullet^{v_0} \ar@(ul,ur)^{e_0} \ar[r]^e & \bullet^v\\
\bullet^{e_2e_1} \ar[ur]_{\overline{e_2e_1}}&& }$$ represents the graph $E(H)$.
\end{exas}

\begin{lem}[{cf. \cite[Theorem 3.8]{ar:cocka}}] Let $K$ be a field and
$E$ a graph, and let $H$ be a hereditary subset of $E^0$. Suppose
\[(E^0\setminus H, r^{-1}_E(E^0\setminus H), s_E, r_E)\] is a finite acyclic graph and
$v\geq H$ for all $v\in E^0\setminus H$. Assume furthermore that the set
$s^{-1}(E^0\setminus H) \cap r^{-1}(H)$ is finite. Then $L_K(E)\cong L_K(E(H))$.
\end{lem}
\begin{proof}
Let us consider an $K$-algebra homomorphism $$\phi: L_K(E(H))\longrightarrow
L_K(E)$$ given on the generators of the free $K$-algebra $K\langle v, e, e^* \mid v\in E(H)^0,
e\in E(H)^1\rangle$ as follows: For $v\in E(H)^0$ define
\begin{equation*}
\phi(v)=  \left\{
\begin{array}{lcl}
v&  & \text{if } v\in H , \\
\alpha\alpha^*&  & \text{if } v = \alpha \in F(H)%
\end{array}%
\right.
\end{equation*}%
and for $e\in E(H)^1$ define
\begin{equation*}
\phi(e)=  \left\{
\begin{array}{lcl}
e&  & \text{if } e\in s^{-1}_E(H), \\
\alpha&  & \text{if } e = \overline{\alpha} \in \overline{F}(H)%
\end{array}%
\right.
\end{equation*}%
and \begin{equation*}
\phi(e^*)=  \left\{
\begin{array}{lcl}
e^*&  & \text{if } e\in s^{-1}_E(H), \\
\alpha^*&  & \text{if } e = \overline{\alpha} \in \overline{F}(H)%
\end{array}%
\right.
\end{equation*}%
To be sure that in a such manner defined map $\phi: L_K(E(H))\longrightarrow L_K(E)$, indeed, provides
us with the desired ring homomorphism, we only need to verify that all following elements:

$vw - \delta_{v, w}v$ for all $v, w\in E(H)^0$,

$s_{E(H)}(e)e - e$ and $e-er_{E(H)}(e)$ for all $e\in E(H)^1$,

$r_{E(H)}(e)e^*-e^*$ and $e^*-e^*s_{E(H)}(e)$ for all $e\in E(H)^1$,

$e^*f - \delta_{e, f}r_{E(H)}(e)$ for all $e, f\in E(H)^1$,

$v - \sum_{e\in (s_{E(H)})^{-1}(v)}ee^*$ for a regular vertex $v\in E(H)^0$\\
are in the kernel of $\phi$. But the latter can be established right away by repeating verbatim
the corresponding obvious arguments in the proof of \cite[Theorem 3.8]{ar:cocka}.

Similar to the proof of Lemma 3.4 for injectivity of $\varphi$ and use the note that if $c$ is a cycle in $E(H)$
without exits, then since  the cycles in $E(H)$ come from cycles in $E$ all lying in the subgraph
given by $(H, s^{-1}_E(H), s_E, r_E)$, we must have $\phi(c) = c$ is a cycle in $E$ without exits, we get
immediately that $\phi$ is injective. Also, $\phi$ is surjective, by repeating verbatim the corresponding argument in
the proof of \cite[Theorem 3.8]{ar:cocka}. Therefore, $\phi$ is an isomorphism, finishing the proof.
\end{proof}

\begin{defn}[{\cite[Definition 3.9]{ar:cocka}}]
Let $E$ be a graph, let $v_0\in E^0$ be a vertex, and let $n$ be a
positive integer. Define a graph $E'(v_0, n)$ as follows:
\[E'(v_0, n)^0 = E^0 \cup \{v_1, v_2, ..., v_n\}\]
\[E'(v_0, n)^1 = E^1 \cup \{e_1, e_2, ..., e_n\}\]
where $r_{E'(v_0, n)}$ and $s_{E'(v_0, n)}$ extends $r_E$ and $s_E$ respectively and
$r_{E'(v_0, n)}(e_i) = v_{0}$ and $s_{E'(v_0, n)}(e_i) = v_{i}$ for all $i = 1, ..., n$.
\end{defn}

\begin{cor}[{cf. \cite[Corollary 3.10]{ar:cocka}}]
Let $K$ be a field and $E$ a graph, let $v_0\in E^0$ be a vertex, and let $n$ be a
positive integer. Then $L_K(E(v_0,n))\cong L_K(E'(v_0,n))$.
\end{cor}
\begin{proof} It is not hard to see that $E^0$ is a hereditary subset of $E(v_0,n)^0$, and $E(v_0,n)(E^0)$
is isomorphic to the graph $E'(v_0,n)$. Therefore, by Lemma 3.7, we immediately get the statement.
\end{proof}

Let $E$ be a finite graph. If $E$ is acyclic, then repeated application of the source
elimination process to $E$ yields the empty graph. On the other hand, if $E$ contains a cycle,
then repeated application of the source elimination process will yield a source-free graph
$E_{sf}$ which necessarily contains a cycle.

Consider the sequence of graphs which arises in some step-by-step process of source
eliminations
\[E= E_0\rightarrow E_1\rightarrow\cdots\rightarrow E_i\rightarrow\cdots
\rightarrow E_t = E_{sf}.\] To avoid defining a graph to be the empty set, we define
$E_{sf}$ to be the graph $E_{triv}$ (consisting of one vertex and no edges) in case
$E_{t-1} = E_{triv}$.

Although there in general are many different orders in which a step-by-step source
elimination process can be carried out, the resulting source-free subgraph $E_{sf}$ is always the same
(see, e.g., \cite[Lemma 3.13]{anp:lpahugn}).

\begin{thm}
Let $E$ be a finite graph and $K$ a field. Let \[E= E_0\rightarrow E_1\rightarrow\cdots\rightarrow E_i\rightarrow\cdots
\rightarrow E_t = E_{sf}\] be a sequence of graphs which arises in some step-by-step process of source
eliminations. Then the following statements are true:

(1) If some $E_i\ (0\leq i\leq t)$ contains an isolated vertex, then $L_K(E)$ has Invariant Basis Number;

(2) If no $E_i\ (0\leq i\leq t)$ contains an isolated vertex, then there exists a finite source-free graph $F$ satisfying the
following conditions:
\begin{itemize}
\item[(i)] $L_K(E)\cong L_K(F);$
\item[(ii)] The cycles of $F$ without exits are in one-to-one correspondence to the cycles of $E$ without exits;
\item[(iii)] The source cycles of $F$ are in one-to-one correspondence to the source cycles of $E_{sf}$.
\end{itemize}
\end{thm}
\begin{proof}
(1) Assume first that $E_i$ contains an isolated vertex for some $i$. Let $j$ denote
the minimal such $i$. Then, at each step of the source elimination process
\[E= E_0\rightarrow E_1\rightarrow\cdots\rightarrow E_j\] the source which is being eliminated
is not an isolated vertex.

It is not hard to check that $E^0_j$ is a hereditary subset of
$E^0$, the finite graph $(E^0\setminus E^0_j, r^{-1}_E(E^0\setminus E^0_j), s_E, r_E)$ is acyclic and
$v\geq E^0_j$ for all $v\in E^0\setminus E^0_j$. Therefore, by Lemma 3.7, $L_K(E)\cong L_K(E(E^0_j))$.
As was mentioned earlier, $E(E^0_j)$ is just the graph $(E^0_j, s^{-1}_E(E^0_j), s_E, r_E)$ together with a
source for each $\alpha\in F(E^0_j)$ with exactly one edge from $\alpha$ to $r_E(\alpha)$.
Let $v$ be the isolated vertex in $E_j$ and $n$ the number of paths in $E$ ending in $v$. Consider
the subgraph $H = (H^0, H^1)$ of $E(E^0_j)$ as follows: \[H^0 := \{v, s_{E(E^0_j)}(f)\mid f\in
r^{-1}_{E(E^0_j)}(v)\} \text{ and } H^1 := r^{-1}_{E(E^0_j)}(v).\] We then have that $E(E^0_j)
= H \sqcup E(E^0_j)\setminus H$, and hence, $$L_K(E(E^0_j))\cong L_K(H)\oplus L_K(E(E^0_j)\setminus H).$$
It shows that there is a natural surjection from $L_K(E(E^0_j))$ onto $L_K(H)$.

On the other hand, by Corollary 3.9, $L_K(H)$ is isomorphic to $L_K(A_{n+1})$, where
\[A_{n+1}= \xymatrix{
    \bullet^{v_n} \ar@{->}[r]^{e_n} & \bullet^{v_{n-1}} \ar@{->}[r]^{e_{n-1}} & \ldots & \bullet^{v_1} \ar@{->}[r]^{e_{1}} & \bullet^{v}.
    }\] It is also well-known that $L_K(A_{n+1}) \cong M_{n+1}(K)$, so $L_K(H)\cong M_{n+1}(K)$.
This implies that $L_K(H)$ has Invariant Basis Number, and therefore,
$L_K(E(E^0_j))$ has Invariant Basis Number, by \cite[Remark 1.5]{l:lomar}.

(2) Suppose that no $E_i$ contains an isolated vertex. Notice that $E^0_{sf}$ is a hereditary subset of
$E^0$, that $(E^0\setminus E^0_{sf}, r^{-1}_E(E^0\setminus E^0_{sf}), s_E, r_E)$ is a finite  acyclic graph, and
that for each  $v\in E^0\setminus E^0_{sf}$ there exists a path in $E$ from $v$ to $E^0_{sf}$. Therefore,
by Lemma 3.7, $L_K(E)\cong L_K(E(E^0_{sf}))$. We can apply Corollary 3.9 and Lemma 3.4 as many times as needed
(but infinitely many times) to get a finite source-free graph $F$ such that $L_K(E)\cong L_K(F)$.

Note that the cycle structure of $E(E^0_{sf})$ and $E$ are determined by the cycle structure of $E_{sf}$ and vice versa,
that the isomorphisms, defined in Lemmas 3.4 and 3.7, and Corollary 3.9, bring a cycle without exits to a
cycle without exits. Moreover, Lemma 3.4 allows one to remove heads of finite length while preserving isomorphism classes.
From these notes, we immediately get the statements (ii) and (iii), finishing the proof.
\end{proof}

A corollary of Theorem 3.10, we have reduced the question to source-free graphs. In light of this note,
we next provides some certain classes of finite graphs for which the Leavitt path algebra having Invariant
Basis Number. We first consider finite graphs containing a source cycle, which is given in \cite{anp:lpahugn}
to study Leavitt path algebras having Unbounded Generating Number. In fact, the following result follows immediately
from \cite[Theorem 3.16]{anp:lpahugn}, but we want to express another proof in terms of Theorem 2.5.

\begin{cor}
Let $E$ be a finite graph and $K$ a field. Let \[E= E_0\rightarrow E_1\rightarrow\cdots\rightarrow E_i\rightarrow\cdots
\rightarrow E_t = E_{sf}\] be a sequence of graphs which arises in some step-by-step process of source
eliminations. Then, if $E_i$ contains an isolated vertex (for some $0\leq i\leq t$), or $E_{sf}$ contains contains a source cycle,
then $L_K(E)$ has Invariant Basis Number.
\end{cor}
\begin{proof} We denote $E^0$ by $\{v_1, v_2, ..., v_h\}$, in such a way that the non-sink vertices of $E$
appear as $v_1, ..., v_z$. We then have that
$$[A^t_E - J_E \ \ b]=\left(\begin{tabular}{cccccccc}
$a_{11}-1$&$a_{21}$&...&$a_{z1}$&$0$&...&$0$&1\\
$a_{12}$&$a_{22}-1$&...&$a_{z2}$&$0$&...&$0$&1\\
.&.&.&.&.&.&.\\
$a_{1z}$&$a_{2z}$&...&$a_{zz} -1$&$0$&...&$0$&1\\
$a_{1(z+1)}$&$a_{2(z+1)}$&...&$a_{z(z+1)}$&$0$&...&$0$&1\\
.&.&.&.&.&.&.\\
$a_{1h}$&$a_{2h}$&...&$a_{zh}$&0&...&0&1\\
\end{tabular}\right).$$

Assume first that $E_i$ contains an isolated vertex for some $i$. By Theorem 3.10 (1), $L_K(E)$ has Invariant Basis Number.

On the other hand, suppose that no $E_i$ contains an isolated vertex. By Theorem 3.10 (2),
we may assume without loss of generality that $E$ is a source-free graph, that means,
$E = E_{sf}$. Then, by our hypothesis, $E$ contains a source cycle $c$, i.e, $|r^{-1}(v)| = 1$ for all $v\in c^0$.
By renumbering vertices if necessary, we may assume without loss of generality that $c^0= \{v_1, ..., v_p\}$.
(Note that, as each vertex in $c^0$ emits at least one edge, we have that each of $\{v_1, ..., v_p\}$ is a regular
vertex.) The condition $|r^{-1}(v)| = 1$ then yields:
\begin{itemize}
\item[-] $a_{i(i+1)} =1$ for $1\leq i\leq p-1$;
\item[-] $a_{p1} = 1$;
\item[-] $a_{j(i+1)} =0$ for $1\leq i\leq p-1$ and $j\neq i\ (1\leq j\leq h)$;
\item[-] $a_{j1} = 0$ if $j\neq p\ (1\leq j\leq h)$.
\end{itemize}
If $p=1$ (i.e., if $c$ is a loop), then $a_{1,1} = 1$, and the matrix $[A^t_E - J_E \ \ b]$ becomes
$$[A^t_E - J_E \ \ b]=\left(\begin{tabular}{cccccccc}
$0$&$0$&...&$0$&$0$&...&$0$&1\\
$a_{12}$&$a_{22}-1$&...&$a_{z2}$&$0$&...&$0$&1\\
.&.&.&.&.&.&.\\
$a_{1z}$&$a_{2z}$&...&$a_{zz} -1$&$0$&...&$0$&1\\
$a_{1(z+1)}$&$a_{2(z+1)}$&...&$a_{z(z+1)}$&$0$&...&$0$&1\\
.&.&.&.&.&.&.\\
$a_{1h}$&$a_{2h}$&...&$a_{zh}$&0&...&0&1\\
\end{tabular}\right).$$
This implies that \[\mathrm{ramk}(A^t_E -J_E) < \mathrm{rank}([A^t_E -J_E\ \ b]),\]
so $L_K(E)$ has Invariant Basis Number, by Theorem 2.5.

If $p \geq 2$, then using the noted information about the $a_{ij}$, the $p$ first rows of the matrix $[A^t_E -J_E\ \ b]$ can be written as:
$$\left(\begin{tabular}{cccccccccc}
$-1$&$0$&$0$&...&$0$&$1$&$0$&...&$0$&$1$\\
$1$&$-1$&$0$&...&$0$&$0$&$0$&...&$0$&$1$\\
$0$&$1$&$-1$&...&$0$&$0$&$0$&...&$0$&$1$\\
.&.&.&.&.&.&.&.&.\\
$0$&$0$&$0$&...&$1$&$-1$&$0$&...&$0$&$1$\\
$a_{1(p+1)}$&$a_{2(p+1)}$&$a_{3(p+1)}$&...&$a_{(p-1)(p+1)}$&$a_{p(p+1)}$&$a_{(p+1)(p+1)}$&...&$0$&$1$\\
.&.&.&.&.&.&.\\
$a_{1h}$&$a_{2h}$&$a_{3h}$&...&$a_{(p-1)h}$&$a_{ph}$&$a_{(p+1)h}$&...&$0$&$1$\\
\end{tabular}\right).$$
We add all rows $i\ (2\leq i\leq p)$ from the first row in the matrix $[A^t_E -J_E\ \ b]$, which yields the
equivalent matrix:
$$\left(\begin{tabular}{cccccccccc}
$0$&$0$&$0$&...&$0$&$0$&$0$&...&$0$&$p$\\
$1$&$-1$&$0$&...&$0$&$0$&$0$&...&$0$&$1$\\
$0$&$1$&$-1$&...&$0$&$0$&$0$&...&$0$&$1$\\
.&.&.&.&.&.&.&.&.\\
$0$&$0$&$0$&...&$1$&$-1$&$0$&...&$0$&$1$\\
$a_{1(p+1)}$&$a_{2(p+1)}$&$a_{3(p+1)}$&...&$a_{(p-1)(p+1)}$&$a_{p(p+1)}$&$a_{(p+1)(p+1)}$&...&$0$&$1$\\
.&.&.&.&.&.&.\\
$a_{1h}$&$a_{2h}$&$a_{3h}$&...&$a_{(p-1)h}$&$a_{ph}$&$a_{(p+1)h}$&...&$0$&$1$\\
\end{tabular}\right).$$ This implies that \[\mathrm{ramk}(A^t_E -J_E) < \mathrm{rank}([A^t_E -J_E\ \ b]),\]
so $L_K(E)$ has Invariant Basis Number, by Theorem 2.5, thus completing the
proof of the corollary.
\end{proof}

We conclude these and paper by considering the class of finite graphs without two distinct cycles
have a common vertex. Interestingly, in \cite{aajz:lpaofgkd} the authors showed that the Leavitt path
algebra of such a graph has finite Gelfand-Kirillov dimension and vice versa. The following corollary
shows that the Leavitt path algebra of such a graph has Invariant Basis Number.

\begin{cor}
Let $K$ be a field and $E$ a finite graph without two distinct cycles have a common vertex.
Then $L_K(E)$ has Invariant Basis Number.
\end{cor}
\begin{proof} By Theorem 3.10, we may assume without loss of generality that $E$ is a source-free graph.
Then, by our hypothesis, we immediately get that $E$ contains a source cycle, and hence, $L_K(E)$
has Invariant Basis Number, finishing the proof.
\end{proof}

\vskip 0.5 cm \vskip 0.5cm {

\end{document}